\documentclass[a4paper]{amsart}
\usepackage{amscd}
\usepackage{amsthm}
\usepackage{graphicx}
\usepackage{amsmath}
\usepackage{amsfonts}
\usepackage{amssymb}

\newtheorem{theorem}{Theorem}

\newtheorem{lemma}[theorem]{Lemma}

\begin{document}
\title[Convolution Operators with Singular Measures of Fractional Type]{Convolution Operators with Singular Measures of Fractional Type on the Heisenberg Group}
\author{Tom\'as Godoy and Pablo Rocha}
\address{Universidad Nacional de C\'ordoba, FaMAF-UNC, C\'ordoba, 5000 C\'ordoba, Argentina}
\email{godoy@famaf.unc.edu.ar, \ rp@famaf.unc.edu.ar}
\thanks{\textbf{2.010
Math. Subject Classification}: 423A80, 42A38}
\thanks{\textbf{Key
words and phrases}: Singular measures, group Fourier transform, Heisenberg group, convolution operators}
\thanks{Partially supported by Conicet and Secyt-UNC}
\maketitle
\begin{abstract}
We consider the Heisenberg group $\mathbb{H}^{n}=\mathbb{C}^{n} \times \mathbb{R}$.
Let $\mu_{\gamma}$ be the fractional Borel measure on $\mathbb{H}^{n}$ defined by
$$ \mu_{\gamma}(E) = \int_{\mathbb{C}^{n}}\chi_{E}\left(w,\varphi(w)\right) \prod_{j=1}^{n} \eta_j \left( |w_j|^{2} \right) | w_j |^{-\frac{\gamma}{n}}dw,$$
where $0 < \gamma < 2n$, $\varphi(w) = \sum\limits_{j=1}^{n} a_{j} \left\vert w_{j}\right\vert^{2}$, $w=(w_{1},...,w_{n}) \in \mathbb{C}^{n}$,
$a_{j} \in \mathbb{R}$, and $\eta_j \in C_{c}^{\infty}(\mathbb{R})$. In this paper we study the set of pairs $(p,q)$ such that the right convolution operator with $\mu_{\gamma}$ is bounded from $L^{p}(\mathbb{H}^{n})$ into $L^{q}(\mathbb{H}^{n})$.

\end{abstract}

\section{Introduction}

Let $\mathbb{H}^{n} = \mathbb{C}^{n} \times \mathbb{R}$ be the Heisenberg group with group law $\left( z,t\right) \cdot \left(w,s\right)
=\left( z+w,t+s+\left\langle z,w\right\rangle \right) $ where $\langle z,w \rangle = \frac{1}{2} Im(\sum \limits_{j=1}^{n}z_{j} \cdot \overline{w_{j}})$. For $x=(x_{1},...,x_{2n})\in \mathbb{R}^{2n}$, we write $x=(x^{\prime },x^{\prime \prime })$ with
$x^{\prime }\in \mathbb{R}^{n}$, $x^{\prime \prime }\in \mathbb{R}^{n}$. So, $\mathbb{R}^{2n}$ can be identified with $\mathbb{C}^{n}$ via the map
$\Psi (x^{\prime },x^{\prime \prime })=x^{\prime }+ix^{\prime \prime }$. In this setting the form $\langle z,w \rangle $ agrees with the standard symplectic form on $\mathbb{R}^{2n}$. Thus $\mathbb{H}^{n}$ can be viewed as $\mathbb{R}^{2n} \times \mathbb{R}$
endowed with the group law
$$
\left( x,t\right) \cdot \left( y,s\right) =\left( x+y,t+s+\frac{1}{2} B(x,y) \right)
$$
where the symplectic form  $B$ is given by $B(x,y)=\sum \limits_{j=1}^{n}\left( y_{n+j}x_{j}-y_{j}x_{n+j}\right) $, with $x=(x_{1},...,x_{2n})$ and $y=(y_{1},...,y_{2n})$, with neutral element $(0,0)$, and with inverse $\left( x,t\right) ^{-1}=\left(-x,-t\right) $.

Let $\varphi :\mathbb{R}^{2n}\rightarrow \mathbb{R}$ be a measurable function, and let $\mu_{\gamma} $ be the fractional Borel measure on $\mathbb{H}^{n}=\mathbb{R}^{2n}\times \mathbb{R}$ supported on the graph of $\varphi$, given by
\begin{equation}
\left\langle \mu_{\gamma} ,f\right\rangle
=\int\limits_{\mathbb{R}^{2n}}f\left( w,\varphi \left( w\right)
\right)  \prod_{j=1}^{n} \eta_j \left( |w_j|^{2} \right) \left\vert w_j \right\vert ^{-\frac{\gamma}{n} }dw
  \label{mu2}
\end{equation}
with $0< \gamma < 2n$, and where the $\eta _{j}$'s are functions in $C_{c}^{\infty }(\mathbb{R})$ such that $0\leq \eta _{j}\leq 1$, $\eta _{j}(t)\equiv 1$ if $t\in \lbrack -1,1]$ and $supp(\eta _{j})\subset (-2,2)$.

Let $T_{\mu_{\gamma} }$ be the right convolution operator by $\mu_{\gamma} $, defined by
\begin{equation}
T_{\mu_{\gamma} }f\left( x,t\right) =\left( f\ast \mu_{\gamma} \right) \left(x,t\right) =\int_{\mathbb{R}^{2n}}f\left( \left( x,t\right) \cdot
\left( w,\varphi \left( w\right) \right) ^{-1}\right) \prod_{j=1}^{n} \eta_j \left( |w_j|^{2} \right) \left\vert w_j \right\vert ^{-\frac{\gamma}{n}} dw.   \label{tmu}
\end{equation}
We are interested in studying the type set
$$
E_{\mu_{\gamma} }=\left\{ \left( \frac{1}{p},\frac{1}{q}\right) \in \left[0,1\right] \times \left[ 0,1\right] :\left\Vert T_{\mu_{\gamma}}\right\Vert _{L^{p}-L^{q}}<\infty \right\}
$$
where the $L^{p}$ - spaces are taken with respect to the Lebesgue measure on $\mathbb{R}^{2n+1}$. We say that that the measure $\mu_{\gamma}$ defined in (\ref{mu2}) is $L^{p}$-\textit{improving} if $E_{\mu_{\gamma}}$ does not reduce to the diagonal $1/p=1/q$.

This problem is well known if in (\ref{tmu}) we consider $\gamma =0$ and replace the Heisenberg group
convolution with the ordinary convolution in $\mathbb{R}^{2n+1}$. If the
graph of $\varphi$ has non-zero Gaussian curvature at each point, a theorem
of Littman (see \cite{littman}) implies that $E_{\nu }$ is the closed
triangle with vertices $(0,0)$, $(1,1)$, and $\left( \frac{2n+1}{2n+2},\frac{%
1}{2n+2}\right) $ (see \cite{oberlin}). A very interesting survey of results
concerning the type set for convolution operators with singular measures can
be found in \cite{ricci}.\\
Returning to our setting $\mathbb{H}^{n}$, in \cite{secco} and \cite{secco2} S. Secco obtains $L^{p}$-\textit{improving} properties of measures
supported on curves in $\mathbb{H}^{1}$, under certain assumptions. In \cite{ricci2} F. Ricci and E. Stein showed that the type set of the
measure given by (\ref{mu2}), for the case $\varphi(w)=0$, $\gamma =0$ and $n=1$, is the triangle with vertices
$(0,0),$ $(1,1),$ and $\left( \frac{3}{4},\frac{1}{4}\right)$. In \cite{G-R}, the authors adapt the work of Ricci and Stein for the case of manifolds quadratic hypersurfaces in $\mathbb{R}^{2n+1}$, there we also give some examples of surfaces with degenerate curvature at the origin.

We observe that if $\left( \frac{1}{p},\frac{1}{q}\right) \in E_{\mu_{\gamma}}$ then
\begin{equation}
p\leq q, \,\,\,\,\,\,\,\,\,\,\,\, \frac{1}{q}\geq \frac{2n+1}{p}-2n, \,\,\,\,\,\,\,\,\,\,\,\,  \frac{1}{q}\geq\frac{1}{(2n+1)p}. \label{restricciones}
\end{equation}

Indeed, the first inequality follows from Lemma 3 in \cite{G-R}, replacing the sets $A_{\delta}$ and $F_{\delta, x}$ in the proof of Lemma 4 in \cite{G-R} by the sets
$$
A'_{\delta }=\left\{ (x,t)\in \mathbb{R}^{2n}\times
\mathbb{R}:x\in \widetilde{D}\wedge \left\vert t-\varphi (x)\right\vert \leq
\frac{\delta }{4}\right\}
$$
and
$$
F'_{\delta ,x}=\left\{ y\in \widetilde{D}:\left\Vert x-y\right\Vert _{\mathbb{R}%
^{2n}}\leq \frac{\delta }{4n(1+\left\Vert \nabla \varphi \mid
_{supp(\eta )}\right\Vert _{\infty })}\right\}
$$ where $\widetilde{D}$ is a closed disk in $\mathbb{R}^{2n}$ contained in the unit disk centered in the origin such that the origin not belongs to $\widetilde{D}$, we observe that the argument utilized in the proof of Lemma 4 in \cite{G-R} works in this setting so we get the others two inequalities.

Since $0 < \gamma < 2n$ it is clear that $\| T_{\mu_{\gamma}} f \|_{p} \leq c \|f\|_{p}$ for all Borel function $f \in L^{p}(\mathbb{H}^{n})$ and all $1 \leq p \leq \infty$, so $(\frac{1}{p}, \frac{1}{p}) \in E_{\mu_{\gamma}}$.

In Lemma 4, section 2 below, we obtain the following necessary condition for the pair $(\frac{1}{p}, \frac{1}{q})$ to be in $E_{\mu_{\gamma}}$:
$$\frac{1}{q}\geq \frac{1}{p}-\frac{2n-\gamma }{2n+2}.$$

Let $D$ be the point of intersection, in the $(\frac{1}{p}, \frac{1}{q})$ plane, between the lines $\frac{1}{q}=\frac{2n+1}{p}
-2n$, $\frac{1}{q}=\frac{1}{p}-\frac{2n-\gamma }{2n+2}$ and let $D^{\prime }$ be its symmetric with respect to the non principal diagonal. So
\[
D=\left( \frac{4n^{2}+2n+\gamma }{2n(2n+2)},\frac{2n+(2n+1)\gamma }{2n(2n+2)}%
\right) =\left( \frac{1}{p_{D}},\frac{1}{q_{D}}\right) \text{ y\
}D^{\prime }=\left( 1-\frac{1}{q_{D}},1-\frac{1}{p_{D}}\right).
\]
Thus $E_{\mu_{\gamma}}$ is contained in the closed trapezoid with vertices $(0,0)$, $(1,1)$, $D$ and $D^{\prime }$.

Finally, let $C_{\gamma}$ be the point of intersection of the lines
$\frac{1}{q}=1- \frac{1}{p}$ and $\frac{1}{q}= \frac{1}{p}- \frac{2n-\gamma}{2n+2}$, thus $C_{\gamma}= \left( \frac{4n+2-\gamma}{2(2n+2)},
\frac{2+\gamma}{2(2n+2)} \right)$.

In section 3 we prove the following results:

\begin{theorem} If $\mu_{\gamma}$ is the fractional Borel measure defined in (\ref{mu2}), supported on the graph of the function $\varphi
(w)= \sum_{j=1}^{n} a_j \left\vert w_j \right\vert ^{2}$, with $n \in \mathbb{N}$, $a_j \in \mathbb{R}$ and $w_j \in \mathbb{R}^{2}$, then the interior of the type set $E_{\mu_{\gamma}}$ coincide with the interior of the trapezoid
with vertices $(0,0)$, $(1,1)$, $D$ and $D^{\prime }$. Moreover the semi-open segments $\left[(1,1);(p_{D}^{-1},q_{D}^{-1}) \right)$ and
$\left[(0,0);(1-q_{D}^{-1},1-p_{D}^{-1}) \right)$ are contained in $E_{\mu_{\gamma}}$.
\end{theorem}

\begin{theorem} Let $\mu_{\gamma}$ be a fractional Borel measure as in Theorem 1. Then $C_{\gamma} \in E_{\mu_{\gamma}}$.
\end{theorem}

 Let $\widetilde{\mu}_{\gamma}$ be the Borel measure given by
\begin{equation}
\left\langle \widetilde{\mu}_{\gamma} ,f\right\rangle
=\int\limits_{\mathbb{R}^{2n}}f\left( w, |w|^{2m}
\right)  \eta \left( |w|^{2} \right) | w |^{\gamma }dw,
  \label{mu3}
\end{equation}
where $m \in \mathbb{N}_{\geq 2}$, $\gamma = \frac{2(m-1)}{(n+1)m}$, and $\eta$ is a function in $C_{c}^{\infty }(\mathbb{R})$ such that $0\leq \eta \leq 1$, $\eta(t)\equiv 1$ if $t\in \lbrack -1,1]$ and $supp(\eta)\subset (-2,2)$. \\
In a similar way we characterize the type set of the Borel measure $\widetilde{\mu}_{\gamma}$ supported on the graph of the function $\varphi(w) = |w|^{2m}$. In fact we prove

\begin{theorem} Let $\widetilde{\mu}_{\gamma}$ be the Borel measure defined in (\ref{mu3}) with $\gamma = \frac{2(m-1)}{(n+1)m}$, where $n \in \mathbb{N}$
and $m \in \mathbb{N}_{\geq 2}$. Then the type set $E_{\widetilde{\mu}_{\gamma}}$ is the closed triangle with vertices
\[
A=\left( 0,0\right) ,\qquad B=\left( 1,1\right) ,\qquad C=\left( \frac{2n+1}{
2n+2},\frac{1}{2n+2}\right).
\]
\end{theorem}
This result improves to the one obtained in Theorem 2 in \cite{G-R}.

\qquad

Throughout this work, $c$ will denote a positive constant not necessarily the same at each occurrence.

\section{Auxiliary results}

\begin{lemma} Let $\mu_{\gamma}$ be the fractional Borel measure defined by (\ref{mu2}), where $\varphi(w) = \sum_{j=1}^{n} a_j |w_j|^{2}$ and $0 < \gamma < 2n$. If $\left(\frac{1}{p},\frac{1}{q}\right) \in E_{\mu_{\gamma}}$, then $\frac{1}{q}\geq \frac{1}{p}-\frac{2n-\gamma}{2n+2}$.
\end{lemma}

\begin{proof} For $0<\delta \leq 1$, we define $Q_{\delta}=D_{\delta}\times \left[-(4M+n)\delta^{2},(4M+n)\delta^{2}\right]$, where
$D_{\delta}=\left\{x \in \mathbb{R}^{2n}: \left\Vert x \right\Vert \leq \delta \right\}$ and $M=\max \left\{\left\vert \varphi(y)\right\vert:y \in D_{1} \right\}.$ We put
$$A_{\delta} = \left\{(x,t) \in D_{\frac{\delta}{2}}\times \mathbb{R}: \left\vert t- \varphi(x) \right\vert \leq 2M\delta^{2}\right\}.$$
Let $f_{\delta} = \chi_{Q_{\delta}}$. We will prove first that $\left\vert (f_{\delta} \ast \mu_{\gamma})(x,t) \right\vert\geq c \delta^{2n-\gamma}$ for all $(x,t) \in A_{\delta}$, where $c$ is a constant independent of $\delta$. \\
If $(x,t) \in A_{\delta}$, we have that
\begin{equation}
(x,t)\cdot (y,\varphi(y))^{-1} \in Q_{\delta} \text{\textit{ for all }} y \in D_{\frac{\delta}{2}}; \label{qd}
\end{equation}
indeed, $(x,t)\cdot (y,\varphi(y))^{-1}=\left(x-y,t-\varphi(y)-\frac{1}{2} B(x,y) \right)$,
from the homogeneity of $\varphi$ and since
$\frac{1}{2}\left\vert B(x,y)\right\vert \leq n\left\Vert x\right\Vert _{ \mathbb{R}^{2n}}\left\Vert x-y\right\Vert _{\mathbb{R}^{2n}}$, (\ref{qd}) follows . So
\[
\left\vert(f_{\delta} \ast \mu)(x,t) \right\vert = \int_{\mathbb{R}^{2n}} f_{\delta}\left((x,t)\cdot(y,\varphi(y))^{-1} \right) \prod_{j=1}^{n} \eta_j(|y_j|^{2}) \left\vert y_j \right\vert^{-\frac{\gamma}{n}} dy
\]
\[
\geq \int_{D_{\frac{\delta}{2}}} \left\vert y \right\vert^{-\gamma} \prod_{j=1}^{n} \eta_j(|y_j|^{2}) dy
=\int_{D_{\frac{\delta}{2}}}\left\vert y \right\vert^{-\gamma} dy=c \delta^{2n-\gamma}
\]
for all $(x,t) \in A_{\delta}$ and all $0 < \delta < 1/2$. Thus
\[
\left\Vert f_{\delta} \ast \mu_{\gamma} \right\Vert _{q}\geq \left(\int_{A_{\delta}} \left\vert f \ast \mu_{\gamma} \right\vert^{q} \right)^{\frac{1}{q}}\geq c\delta ^{2n-\gamma}\left\vert A_{\delta} \right\vert ^{\frac{1}{q}} = c \delta^{2n-\gamma+\frac{1}{q}(2n+2)}.
\]
On the other hand
$\left(\frac{1}{p},\frac{1}{q}\right) \in E_{\mu}$ implies
\[
\left\Vert f_{\delta} \ast \mu_{\gamma} \right\Vert _{q}\leq c \left\Vert f_{\delta} \right\Vert_{p}= c \delta^{\frac{1}{p}(2n+2)},
\]
therefore
$\delta^{2n-\gamma+\frac{1}{q}(2n+2)}\leq c
\delta^{\frac{1}{p}(2n+2)}$ for all $0 < \delta < 1$ small enough, then
$$\frac{1}{q}\geq \frac{1}{p} - \frac{2n-\gamma}{2n+2}.$$
\end{proof}

The following two lemmas deal on certain identities that involve to the Laguerre polynomials.
We recall the definition of these polynomials: the Laguerre polynomials $L^{\alpha}_{n}(x)$ are defined by the formula
\[
L^{\alpha}_{n}(x)= e^{x} \frac{x^{-\alpha}}{n!} \frac{d^{n}}{dx^{n}}(e^{-x} x^{n + \alpha}), \,\,\,\,\,\,\,\, n=0, 1, 2, ...
\]
for arbitrary real number $\alpha > -1$.

\begin{lemma} If $Re(\beta)>-1$, then
$$\int\limits_{0}^{\infty }\sigma ^{\beta }L_{k}^{n-1}\left( \sigma
\right) e^{-\sigma \left( \frac{1}{2}+i\xi \right) }d\sigma
=\frac{1}{k!} \left[\frac{d^{k}}{dr^{k}} \left(
\frac{\Gamma (\beta +1)}{(1-r)^{n} \left(
\frac{1}{2}+\frac{r}{1-r}+i\xi \right)^{\beta +1}}\right) \right]_{r=0},$$
for $n \in \mathbb{N}$ and $k \in \mathbb{N} \cup \{ 0 \}$.
\end{lemma}

\begin{proof} Let $0<\epsilon <1$ be fixed. From the generating function identity (4.17.3) in \cite{Lebedev} p. 77, we have

\begin{equation}
\sum\limits_{j\geq 0}
\sigma ^{\beta }L_{j}^{n-1}\left( \sigma \right) e^{-\sigma \left( \frac{1}{2}+i\xi \right) } r^{j}=\frac{1}{(1-r)^{n}}
\sigma ^{\beta }e^{-\sigma \left( \frac{1}{2}+\frac{r}{1-r}+i\xi
\right)}, \,\,\,\,\,\,\,\,\,\, |r|<1. \label{ident3}
\end{equation}
Since $\left\vert L_{j}^{n-1}\left( \sigma \right) e^{-\sigma
\frac{1}{2}} \right\vert \leq \frac{(j+n-1)!}{j!(n-1)!}$ for all
$\sigma > 0$ (see proposition 4.2 in \cite{thangavelu}), the series in
(\ref{ident3}) is uniformly convergent on the interval
$\left[\epsilon, \frac{1}{\epsilon} \right]$. Integrating on this interval we obtain
$$
\sum\limits_{j\geq 0}\left( \int\limits_{\epsilon }^{\frac{1}{\epsilon }%
}\sigma ^{\beta }L_{j}^{n-1}\left( \sigma \right) e^{-\sigma \left( \frac{1}{%
2}+i\xi \right) }d\sigma \right) r^{j}=\frac{1}{(1-r)^{n}}%
\int\limits_{\epsilon }^{\frac{1}{\epsilon }}\sigma ^{\beta
}e^{-\sigma \left( \frac{1}{2}+\frac{r}{1-r}+i\xi \right) }d\sigma,
$$
so
\begin{equation}
\int\limits_{\epsilon }^{\frac{1}{\epsilon }}\sigma ^{\beta
}L_{k}^{n-1}\left( \sigma \right) e^{-\sigma \left(
\frac{1}{2}+i\xi \right)
}d\sigma =\frac{1}{k!}\left[ \frac{d^{k}}{dr^{k}}\left( \frac{1}{%
(1-r)^{n}}\int\limits_{\epsilon }^{\frac{1}{\epsilon }}\sigma
^{\beta }e^{-\sigma \left( \frac{1}{2}+\frac{r}{1-r}+i\xi \right)
}d\sigma \right)\right]_{r=0}. \label{ident4}
\end{equation}
Now let us computation $\left[ \frac{d^{k}}{dr^{k}}
\left( \int\limits_{\epsilon }^{\frac{1}{\epsilon
}}\sigma^{\beta} e^{-\sigma \left( \frac{1}{2}+\frac{r}{1-r}+i\xi
\right) }d\sigma \right) \right]_{r=0}$. We start to compute first the derivatives of the function $u\rightarrow \int\limits_{\epsilon
}^{\frac{1}{\epsilon }}\sigma ^{\beta }e^{-\sigma u }d\sigma$, where $Re(u) > 0$. We define
$\alpha_{\epsilon}(\sigma)= \sigma u$, $\sigma \in \left[\epsilon,
\frac{1}{\epsilon} \right]$, so
$$\int\limits_{\epsilon }^{\frac{1}{\epsilon
}}\sigma ^{\beta }e^{-\sigma u }d\sigma = u^{-(\beta+1)}
\int\limits_{\alpha_{\epsilon }} z^{\beta }e^{-z }dz$$
to apply the Cauchy's Theorem we have
\begin{equation}
\int\limits_{\epsilon }^{\frac{1}{\epsilon
}}\sigma ^{\beta }e^{-\sigma u }d\sigma = u^{-(\beta+1)} \left[\int\limits_{\epsilon }^{\frac{1}{\epsilon
}} x^{\beta }e^{-x }dx + I_{1}(u,\epsilon) - I_{2}(u,\epsilon)
\right] \label{ident5}
\end{equation}
where
$$I_{1}(u,\epsilon)=\int\limits_{\left[\frac{1}{\epsilon},\frac{1}{\epsilon}u \right]} z^{\beta }e^{-z }dz$$
and
$$I_{2}(u,\epsilon)=\int\limits_{\left[\epsilon,\epsilon u \right]} z^{\beta }e^{-z}dz$$
are line integrals on $\mathbb{C}$. Now we will prove that for each $u_{0} \in
\mathbb{C}$ with $Re(u_{0})>0$ the following identity holds
\begin{equation}
\lim_{\epsilon \rightarrow 0} \left[ \frac{d^{k}}{du^{k}} I_{j}(u, \epsilon) \right]_{u=u_{0}} = 0
\label{ij}
\end{equation}
for $j=1,2$ and all $k \geq 0$. It is easy to check that
$$I_{2}(u,\epsilon)=\epsilon ^{\beta+1} \int\limits_{\left[1, u \right]} z^{\beta }e^{-\epsilon z}dz.$$
Since $Re(\beta)>-1$ we have that $\lim_{\epsilon
\rightarrow 0} I_{2}(u_{0}, \epsilon) = 0$. From the analyticity of the function $z \rightarrow z^{\beta }e^{-\epsilon z}$ on the region
$\{ z: Re(z)>0 \}$ it follows for $k \geq 1$
$$\left[ \frac{d^{k}}{du^{k}}  I_{2}(u, \epsilon) \right]_{u=u_{0}}=\epsilon
^{\beta+1} \left[ \frac{d^{k-1}}{du^{k-1}} u^{\beta} e^{-\epsilon u} \right]_{u=u_{0}},$$ then $\lim_{\epsilon
\rightarrow 0} \left[ \frac{d^{k}}{du^{k}} I_{2}(u, \epsilon) \right]_{u=u_{0}} = 0$ for all $k\geq 0$. \\
Analogously and taking account the rapid decay of the function $z\rightarrow e^{-z}$ on the region $\{z : Re(z)>0 \}$ we obtain
that $\lim_{\epsilon \rightarrow 0} \left[ \frac{d^{k}}{du^{k}} I_{1}(u, \epsilon) \right]_{u=u_{0}} =0$ for all $k\geq 0$,
so (\ref{ij}) follows. To derive in (\ref{ident5}), from the Leibniz's formula and (\ref{ij}) it follows that
\begin{equation}
\lim_{\epsilon \rightarrow 0} \left[ \frac{d^{k}}{du^{k}} \int\limits_{\epsilon
}^{\frac{1}{\epsilon }}\sigma ^{\beta }e^{-\sigma u }d\sigma
\right]_{u=u_{0}} =\Gamma(\beta+1)\left[ \frac{d^{k}}{du^{k}} u^{-(\beta+1)}\right]_{u=u_{0}}. \label{iu}
\end{equation}
Finally, from (\ref{iu}), to apply the chain rule to the function
$r\rightarrow \int\limits_{\epsilon }^{\frac{1}{\epsilon
}}\sigma ^{\beta }e^{-\sigma u(r) }d\sigma$ where
$u(r)=\frac{1}{2}+\frac{r}{1-r}+i \xi$ and the Leibniz's formula give, to do $\epsilon \rightarrow 0$ in (\ref{ident4}), that
$$\int\limits_{0}^{\infty }\sigma ^{\beta }L_{k}^{n-1}\left(
\sigma \right) e^{-\sigma \left( \frac{1}{2}+i\xi \right)
}d\sigma =\lim_{{\epsilon \rightarrow 0}} \int\limits_{\epsilon }^{
\frac{1}{\epsilon }}\sigma ^{\beta }L_{k}^{n-1}\left( \sigma
\right) e^{-\sigma \left( \frac{1}{2}+i\xi \right) }d\sigma
$$ $$=\frac{1}{k!} \left[\frac{d^{k}}{dr^{k}} \left(
\frac{\Gamma (\beta +1)}{(1-r)^{n} \left(
\frac{1}{2}+\frac{r}{1-r}+i\xi \right)^{\beta +1}}\right) \right]_{r=0}.$$
\end{proof}

\begin{lemma} If $Re(\beta) > -1$ and $w(\xi)= -\frac{\frac{1}{2}-i \xi}{\frac{1}{2}+i \xi}$ $(\xi \in \mathbb{R})$, then
\[
\int\limits_{0}^{\infty }\sigma ^{\beta }L_{k}^{n-1}\left(
\sigma \right) e^{-\sigma \left( \frac{1}{2}+i\xi \right) }d\sigma
=\frac{\Gamma(\beta+1)}{\left(\frac{1}{2}+i\xi \right)^{\beta+1}}
\sum_{j+l=k}
\frac{\Gamma(n-1-\beta+j)\Gamma(\beta+1+l)}{\Gamma(n-1-\beta)\Gamma(\beta+1)}
\frac{w(\xi)^{l}}{j!l!},
\]
for $n \in \mathbb{N}$ and $k \in \mathbb{N} \cup \{ 0 \}$.
\end{lemma}

\begin{proof}
We will start find the power series centered at $r=0$ of the following function
$$Q(r)=\frac{1}{(1-r)^{n} \left( \frac{1}{2}+ \frac{r}{1-r}+i \xi \right)^{\beta+1}}, \,\,\,\,\,\,\,\,\,\, |r| < 1.$$
We observe that
$$Q(r)=\frac{1}{(1-r)^{n-\beta-1}\left(\frac{1}{2}+ i\xi+ r\left( \frac{1}{2}-i \xi \right) \right)^{\beta+1} },$$
doing $w=-\frac{\frac{1}{2}-i \xi}{\frac{1}{2}+i \xi}$, we obtain
$$Q(r)=\frac{1}{ \left(\frac{1}{2}+i \xi \right)^{\beta+1} (1-r)^{n-1-\beta}(1-rw)^{\beta+1}}.$$
A simple computation gives
$$(1-r)^{-n+\beta+1}=1+\sum_{j\geq 1}
(n-1-\beta)(n-1-\beta+1)...(n-1-\beta+j-1) \frac{r^{j}}{j!}$$
\begin{equation}
=\sum_{j\geq 0} \frac{\Gamma(n-1-\beta+j)}{\Gamma(n-1-\beta)}
\frac{r^{j}}{j!}. \label{serie2}
\end{equation}
Analogously we have
\[
(1-rw)^{-\beta-1}=\sum_{j\geq 0}
\frac{\Gamma(\beta+1+j)}{\Gamma(\beta+1)} \frac{(rw)^{j}}{j!}.
\]
Thus
$$Q(r)=\frac{1}{\left(\frac{1}{2}+i\xi \right)^{\beta+1}} \left(
\sum_{j+l \geq 0}
\frac{\Gamma(n-1-\beta+j)\Gamma(\beta+1+l)}{\Gamma(n-1-\beta)\Gamma(\beta+1)}
\frac{r^{j+l}w^{l}}{j!l!} \right).$$
Finally, from Lemma 5 it follows
$$\int\limits_{0}^{\infty }\sigma ^{\beta }L_{k}^{n-1}\left(
\sigma \right) e^{-\sigma \left( \frac{1}{2}+i\xi \right) }d\sigma
=\frac{\Gamma(\beta+1)}{\left(\frac{1}{2}+i\xi \right)^{\beta+1}}
\sum_{j+l=k}
\frac{\Gamma(n-1-\beta+j)\Gamma(\beta+1+l)}{\Gamma(n-1-\beta)\Gamma(\beta+1)}
\frac{w^{l}}{j!l!}.$$
\end{proof}

\begin{lemma} If $Re(\beta) > -1$, then
\[
\sum_{j+l=k} \frac{\Gamma(n-1-Re(\beta)+j)) \Gamma(Re(\beta)+1+l)}{\Gamma(n-1-Re(\beta))\Gamma(Re(\beta)+1)} \frac{1}{j!l!}=\frac{(n+k-1)!}{(n-1)!k!},
\]
for $n \in \mathbb{N}$ and $k \in \mathbb{N} \cup \{ 0 \}$.
\end{lemma}

\begin{proof} From $(\ref{serie2})$ it obtains
$$\sum_{j+l=k}
\frac{\Gamma(n-1-Re(\beta)+j))
\Gamma(Re(\beta)+1+l)}{\Gamma(n-1-Re(\beta))\Gamma(Re(\beta)+1)}
\frac{1}{j!l!}$$ $$= \frac{1}{k!} \left[ \frac{d^{k}}{dr^{k}}
(1-r)^{-n+Re(\beta)+1}
(1-r)^{-Re(\beta)-1}\right]_{r=0}=\frac{1}{k!} \left[ \frac{d^{k}}{dr^{k}}
(1-r)^{-n} \right]_{r=0}.$$
Since
$$(1-r)^{-n} =\sum_{j\geq 0} \frac{\Gamma(n+j)}{\Gamma(n)}
\frac{r^{j}}{j!}=\sum_{j\geq 0} \frac{(n+j-1)!}{(n-1)!j!} r^{j},$$
we have $$\sum_{j+l=k} \frac{\Gamma(n-1-Re(\beta)+j))
\Gamma(Re(\beta)+1+l)}{\Gamma(n-1-Re(\beta))\Gamma(Re(\beta)+1)}
\frac{1}{j!l!}=\frac{(n+k-1)!}{(n-1)!k!}.$$
\end{proof}

\section{The main results}

To prove Theorem 1 we will decompose the operator $T_{\mu_{\gamma}}$ of the following way: we consider a family $\left\{T_{\mu_{k}} \right\}_{k \in \mathbb{N}}$ of operators such that
$T_{\mu_{\gamma}} = \displaystyle{\sum_{k \in \mathbb{N}}} T_{\mu_{k}}$, $\left\Vert T_{\mu_{k}} \right\Vert_{1,1} \sim 2^{-k(2n-\gamma)}$ and $\left\Vert T_{\mu_{k}} \right\Vert_{p,q} \sim 2^{k\gamma}\left\Vert T_{\mu_{0}} \right\Vert_{p,q}$ where $T_{\mu_{0}}$ is the operator defined by (\ref{tmu}), taking there $\gamma =0$ and $\varphi(w)=\sum_{j=1}^{n} a_j \left\vert w_j\right\vert^{2}$.
Then Theorem 1 will follow from Theorem 1 in \cite{G-R}, the Riesz-Thorin convexity Theorem and Lemma 4.
\\
${}$
\\
\textit{Proof of Theorem 1.} For each $k \in \mathbb{N}$ we define
\[
A_{k}= \left\{ y=(y_1, ..., y_n) \in (\mathbb{R}^{2})^{n} : 2^{-k} < \left\vert y_j \right\vert \leq 2^{-k+1}, j=1, 2,..., n \right\}
\]
Let $\mu_{k}$ be the fracional Borel measure given by
\[
\mu_{k}(E)= \int_{A_{k}} \chi_{E} \left(y, \varphi(y) \right) \prod_{j=1}^{n} \eta_j \left( |y_j|^{2} \right) | y_j |^{-\frac{\gamma}{n}} dy
\]
and let $T_{\mu_{k}}$ be its corresponding convolution operator, i.e: $T_{\mu_{k}}f=f \ast \mu_{k}$.
Now, it is clear that  $\mu_{\gamma}=\sum_{k}\mu_{k}$ and $\left\Vert T_{\mu_{\gamma}} \right\Vert_{p,q} \leq \sum_{k} \left\Vert T_{\mu_{k}} \right\Vert_{p,q}$.
For $f\geq 0$ we have that
$$
\int f(y,s) d\mu_{k}(y,s) \leq 2^{k\gamma}\int_{\mathbb{R}^{2n}} f\left(y, \varphi(y) \right) \prod_{j=1}^{n} \eta_j \left( |y_j|^{2} \right) dy.$$
Thus $\left\Vert T_{\mu_{k}} \right\Vert_{p,q} \leq c 2^{k\gamma}\left\Vert T_{\mu_{0}} \right\Vert_{p,q}$, from Theorem 1 in \cite{G-R} it follows that
\[
\left\Vert T_{\mu_{k}} \right\Vert_{\frac{2n+2}{2n+1},2n+2} \leq c 2^{k\gamma}.
\]
It is easy to check that $\left\Vert T_{\mu_{k}}
\right\Vert_{1,1}\leq \left\vert \mu_{k} (\mathbb{R}^{2n+1})
\right\vert \sim \int_{A_{k}} \left\vert y \right\vert^{-\gamma} dy= c 2^{-k(2n-\gamma)}.$ \\
For $0< \theta <1$, we define
$$\left(\frac{1}{p_{\theta}}, \frac{1}{q_{\theta}} \right) = \left(\frac{2n+1}{2n+2}, \frac{1}{2n+2} \right) (1-\theta) + (1,1)\theta,$$
by the Riesz convexity Theorem we have
$$
\left\Vert T_{\mu_{k}} \right\Vert_{p_{\theta},q_{\theta}} \leq c
2^{k\gamma(1-\theta)-k(2n-\gamma) \theta}$$
choosing $\theta$ such that $k\gamma(1-\theta)-k(2n-\gamma) \theta = 0$ result $\displaystyle{\sup_{k \in \mathbb{N}}} \left\Vert T_{\mu_{k}}
\right\Vert_{p_{\theta},q_{\theta}} \leq c < \infty$. A simple computation gives $\theta=\frac{2n-\gamma}{2n}$, then $\left( \frac{1}{p_{\theta}}, \frac{1}{q_{\theta}} \right) = \left( \frac{1}{p_{D}}, \frac{1}{q_{D}} \right)$, so $\left\Vert T_{\mu_{k}} \right\Vert_{p_{D},q_{D}} \leq c$,
where $c$ no depend on $k$. Interpolating once again, but now between the points $\left(\frac{1}{p_{D}}, \frac{1}{q_{D}} \right)$ and $(1,1)$ we obtain, for each $0< \tau <1$ fixed
$$
\left\Vert T_{\mu_{k}} \right\Vert_{p_{\tau},q_{\tau}} \leq c 2^{-k(2n-\gamma) \tau},$$
since $\left\Vert T_{\mu_{\gamma}} \right\Vert_{p,q} \leq \sum_{k} \left\Vert T_{\mu_{k}} \right\Vert_{p,q}$ and $0< \gamma < 2n$, it follows that
$$
\left\Vert T_{\mu_{\gamma}} \right\Vert_{p_{\tau},q_{\tau}} \leq c\sum_{k \in \mathbb{N}} 2^{-k(2n-\gamma) \tau} <\infty,$$ by duality we also have
$$\left\Vert T_{\mu_{\gamma}} \right\Vert_{\frac{q_{\tau}}{q_{\tau}-1},\frac{p_{\tau}}{p_{\tau}-1}}\leq c_{\tau} <\infty.$$
Finally, the theorem follows from the Riesz convexity Theorem, the restrictions that appear in (\ref{restricciones}) and Lemma 4.
$\square$

\qquad

To prove Theorem 2, we will consider an auxiliary operator $T_N$, with $N \in \mathbb{N}$ fixed, which will be embedded in an analytic family $T_{N,z}$ of operators on the strip $-\frac{2n-\gamma }{2+\gamma }\leq Re(z)\leq 1$ such that
\begin{equation}
\left\{
\begin{array}{c}
\left\Vert T_{N,z}\left( f\right) \right\Vert _{L^{\infty }\left( \mathbb{H}%
^{n}\right) }\leq A_{z}\left\Vert f\right\Vert _{L^{1}\left( \mathbb{H}%
^{n}\right) }\qquad Re(z)=1 \\
\left\Vert T_{N,z}\left( f\right) \right\Vert _{L^{2}\left( \mathbb{H}%
^{n}\right) }\leq A_{z}\left\Vert f\right\Vert _{L^{2}\left( \mathbb{H}%
^{n}\right) }\qquad Re(z)=-\frac{2n-\gamma }{%
2+\gamma } \label{desig2}
\end{array}
\right.
\end{equation}
where $A_{z}$ will depend admissibly on the variable $z$ and it will not depend on $N$. We denote $T_N = T_{N,0}$. By Stein's theorem on complex interpolation, it will follow that the operator $T_{N}$ will be bounded from $L^{p_{\gamma}}(\mathbb{H}^{n})$ into $L^{p'_{\gamma}}(\mathbb{H}^{n})$, where
$\left(\frac{1}{p_{\gamma}}, \frac{1}{p'_{\gamma}} \right)=C_{\gamma}$, uniformly in $N$. If we see that $T_{N}f(x,t) \rightarrow c T_{\mu_{\gamma}}f(x,t)$ a.e.$(x,t)$ as $N\rightarrow \infty$, then Theorem 2 will follow from Fatou's Lemma.

To prove the second inequality in (\ref{desig2}) we will see that such a
family will admit the expression
\[
T_{N,z}(f)(x,t)=\left( f\ast K_{N,z}\right) (x,t),
\]%
where $K_{N,z}\in L^{1}(\mathbb{H}^{n})$, moreover it is a \textit{polyradial}
function (i.e. the values of $K_{N,z}$ depend on $\left\vert
w_{1}\right\vert ,$...$,\left\vert w_{n}\right\vert $ and $t$). Now our
operator $T_{N,z}$ can be realized as a multiplication of operators via the
group Fourier transform, i.e.
\[
\widehat{T_{N,z}(f)}(\lambda )=\widehat{f}(\lambda )\widehat{K_{N,z}}%
(\lambda )
\]%
where, for each $\lambda \neq 0$, $\widehat{K_{N,z}}(\lambda )$ is an
operator on the Hilbert space $L^{2}(\mathbb{R}^{n})$ given by
\[
\widehat{K_{N,z}}(\lambda )g(\xi )=\int\limits_{\mathbb{H}%
^{n}}K_{N,z}(\varsigma ,t)\pi _{\lambda }(\varsigma ,t)g(\xi )d\varsigma dt.
\]%
It then follows from Plancherel's theorem for the group Fourier transform
that
\[
\left\Vert T_{N,z}f\right\Vert _{L^{2}(\mathbb{H}^{n})}\leq A_{z}\left\Vert
f\right\Vert _{L^{2}(\mathbb{H}^{n})}
\]%
if and only if
\begin{equation}
\left\Vert \widehat{K_{N,z}}(\lambda )\right\Vert _{op}\leq A_{z}
\label{L21}
\end{equation}%
uniformly over\textit{\ }$N$\textit{\ }and\textit{\ }$\lambda \neq 0.$ Since
$K_{N,z}$ is a poliradial integrable function, then by a well known result
of Geller (see Lemma 1.3, p. 213 in \cite{geller}), the operators $\widehat{%
K_{N,z}}(\lambda ):L^{2}(\mathbb{H}^{n})\rightarrow L^{2}(\mathbb{H}^{n})$
are, for each $\lambda \neq 0$, diagonal with respect to a Hermite basis for
$L^{2}(\mathbb{R}^{n})$. This is
\[
\widehat{K_{N,z}}(\lambda )=C_{n}\left( \delta _{\gamma ,\alpha }\nu
_{N,z}(\alpha ,\lambda )\right) _{\gamma ,\alpha \in \mathbb{N}_{0}^{n}}
\]%
where $C_{n}=(2\pi )^{n}$, $\alpha =(\alpha _{1},...,\alpha _{n})$, $\delta _{\gamma ,\alpha }=1$ if $\gamma = \alpha$ and $\delta _{\gamma ,\alpha }=0$ if $\gamma \neq \alpha$, and the
diagonal entries $\nu _{N,z}(\alpha _{1},...,\alpha _{n},\lambda )$ can be
expressed explicitly in terms of the Laguerre transform. We have in fact
\[
\nu_{N,z}(\alpha _{1},...,\alpha _{n},\lambda )=\int\limits_{0}^{\infty
}\,...\,\int\limits_{0}^{\infty }\,K_{N,z}^{\lambda
}(r_{1},...,r_{n})\prod_{j=1}^{n}\left( r_{j}L_{\alpha _{j}}^{0}(\frac{1}{2}%
\left\vert \lambda \right\vert r_{j}^{2})e^{-\frac{1}{4}\left\vert \lambda
\right\vert r_{j}^{2}}\right) \,dr_{1}...dr_{n}
\]%
where $L_{k}^{0}(s)$ are the Laguerre polynomials, i.e. $L_{k}^{0}(s)=%
\sum_{i=0}^{k}\left( \frac{k!}{(k-i)!i!}\right) \frac{(-s)^{i}}{i!}$ and $%
K_{N,z}^{\lambda }(\varsigma )=\int\limits_{\mathbb{R}}K_{N,z}(\varsigma
,t)e^{i\lambda t}dt.$ Now (\ref{L21}) is equivalent to
\[
\left\Vert T_{N,z}f\right\Vert _{L^{2}(\mathbb{H}^{n})}\leq A_{z}\left\Vert
f\right\Vert _{L^{2}(\mathbb{H}^{n})}
\]%
if and only if
\begin{equation}
\left\vert \nu_{N,z}(\alpha _{1},...,\alpha _{n},\lambda )\right\vert \leq
A_{z}  \label{L22}
\end{equation}%
uniformly over\textit{\ }$N$, $\alpha _{j}$\textit{\ }and\textit{\ }$\lambda
\neq 0.$ If $Re(z)=-\frac{2n-\gamma }{2+\gamma }$, in the proof of Theorem 2 we find that (\ref{L22}) holds with $A_{z}$
independent of $N$, $\lambda \neq 0$ and $\alpha _{j}$, and then we obtain
the boundedness on $L^{2}(\mathbb{H}^{n})$ that is stated in (\ref{desig2}).

We consider the family $\{ I_{z} \}_{z \in \mathbb{C}}$ of distributions on $\mathbb{R}$ that arises by analytic continuation of the family
$\{ I_{z} \}$ of functions, initially given when $Re(z)>0$ and $s\in \mathbb{R} \setminus\{ 0 \}$ by
\begin{equation}
I_{z}(s)=\frac{2^{-\frac{z}{2}}}{\Gamma \left( \frac{z}{2}\right) }%
\left\vert s\right\vert ^{z-1}.  \label{iz}
\end{equation}%
In particular, we have $%
\widehat{I_{z}}=I_{1-z}$, also $I_{0}=c\delta $ where $\widehat{\cdot }$
denotes the Fourier transform on $\mathbb{R}$ and $\delta $ is the Dirac
distribution at the origin on $\mathbb{R}$.

Let $H\in S(\mathbb{R)}$ such that $supp(\widehat{H})\subseteq
\left( -1,1\right) $ and $\int \widehat{H}(t)dt=1$. Now we put $\phi
_{N}(t)=H(\frac{t}{N})$ thus $\widehat{\phi _{N}}(\xi )=N\widehat{H}(N\xi )$
and $\widehat{\phi _{N}}\rightarrow \delta $ in the sense of the
distribution, as $N\rightarrow \infty $.

For $z\mathbb{\in C}$ and $N\in \mathbb{N}$, we also define $J_{N,z}$ as the
distribution on $\mathbb{H}^{n}$ given by the tensor products
\begin{equation}
J_{N,z}=\delta \otimes ...\otimes \delta \otimes \left( I_{z}\ast _{\mathbb{R%
}}\widehat{\phi _{N}}\right)  \label{jz}
\end{equation}%
where $\ast _{\mathbb{R}}$ denotes the usual convolution on $\mathbb{R}$ and
$I_{z}$ is the fractional integration kernel given by (\ref{iz}). We observe that
\begin{equation}
J_{N,0}=\delta \otimes
...\otimes \delta \otimes c\widehat{\phi _{N}}\rightarrow \delta \otimes ...\otimes \delta \otimes c\delta   \label{jz2}
\end{equation}
in the sense of the distribution as $N \rightarrow \infty $.

\qquad

\textit{Proof of Theorem 2.} Let $\left\{ T_{N,z}\right\} $ be the family operators on the strip $-\frac{2n-\gamma }{
2+\gamma }\leq Re(z)\leq 1$, given by
\[
T_{N,z}f=f\ast \mu_{\gamma, z}\ast J_{N,z},
\]
where $J_{N,z}$ is given by (\ref{jz}) and $\mu _{\gamma, z}$ by
\begin{equation}
\mu_{\gamma, z}(E)=\int\limits_{\mathbb{R}^{2n}}\chi _{E}\left(
w,\varphi(w)\right) \prod_{j=1}^{n} \eta_j \left( |w_j|^{2} \right)
\left\vert w_j \right\vert^{(z-1) \frac{\gamma}{n}}dw. \label{muz}
\end{equation}
Now (\ref{jz2}) implies that $T_{N, 0}f(x,t)\rightarrow c T_{\mu_{\gamma}}f(x,t)$ a.e.$(x,t)$ as $N\rightarrow \infty $.

\qquad

For $Re(z)=1$ we have
$$\mu _{\gamma, z}\ast J_{N,z}(x,t)=  \left( I_{z}\ast _{\mathbb{R}}\widehat{\phi
_{N}}\right) \left( t-\varphi(x)\right) \prod_{j=1}^{n}\eta_{j} \left( |x_j|^{2} \right) |x_j|^{iIm(z)\frac{\gamma}{n} },$$
so $\left\Vert \mu_{\gamma, z}\ast J_{N,z}\right\Vert _{\infty }\leq c\left\vert \Gamma \left( \frac{z}{2}\right) \right\vert ^{-1}$.
Then, for $Re(z)=1$, we obtain
$$
\left\Vert T_{N,z}f\right\Vert _{\infty }\leq \left\Vert f\ast \mu_{\gamma, z}\ast J_{N,z}\right\Vert _{\infty }\leq \left\Vert f\right\Vert
_{1}\left\Vert \mu_{\gamma, z}\ast J_{N,z}\right\Vert _{\infty }\leq c\left\vert \Gamma \left( \frac{z}{2}\right) \right\vert ^{-1}\left\Vert f\right\Vert _{1}
$$
where $c$ is a positive constant independent of $N$ and $z$.

\qquad

We put $K_{N, z} = \mu _{\gamma, z}\ast J_{N,z}$, for $Re(z)=-\frac{2n-\gamma }{2+\gamma }$ we have that $K_{N,z} \in L^{1}(\mathbb{H}^{n})$. Indeed
\[
K_{N,z}(x,t)=\left( I_{z}\ast_{\mathbb{R}}\widehat{\phi _{N}}\right) \left( t-\varphi( x) \right) \prod_{j=1}^{n}\eta_{j} \left( |x_j|^{2} \right) |x_j|^{(z-1) \frac{\gamma}{n}},
\]
since $0<\gamma <2n$ it follows that $2+Re((z-1)\frac{\gamma}{n})=2-\frac{2n+2}{2+\gamma} \frac{\gamma}{n} > 0$ and so $\prod_{j=1}^{n}\eta_{j} \left( |x_j|^{2} \right) |x_j|^{(z-1) \frac{\gamma}{n}} \in L^{1}(\mathbb{R}^{2n})$, in the proof of Lemma 5 in \cite{G-R} it shows that $\left( I_{z}\ast_{\mathbb{R}}\widehat{\phi _{N}}\right) \in L^{1}(\mathbb{R})$. These two facts imply that $K_{N,z} \in L^{1}(\mathbb{H}^{n})$.
In addition $K_{z,N}$ is a polyradial function. Thus the operator $\widehat{K_{z,N}}(\lambda )$ is diagonal with respect to a Hermite base for $L^{2}(\mathbb{R}^{n})$, and its diagonal entries $\nu_{z,N}(\alpha,\lambda )$, with $\alpha =(\alpha_1, ..., \alpha_n) \in \mathbb{N}_{0}^{n}$, are given by
\[
\nu_{N, z}(\alpha ,\lambda )= \int\limits_{0}^{\infty
}\,...\,\int\limits_{0}^{\infty }\,K_{N,z}^{\lambda
}(r_{1},...,r_{n})\prod_{j=1}^{n}\left( r_{j}L_{\alpha _{j}}^{0}(| \lambda| r_{j}^{2}/2)e^{-\frac{1}{4}\left\vert \lambda
\right\vert r_{j}^{2}}\right) \,dr_{1}...dr_{n}
\]
\begin{equation}
= I_{1-z}(-\lambda
)\phi _{N}(\lambda ) \prod_{j=1}^{n} \int\limits_{0}^{\infty }\eta_{j}(r_j^{2})L_{\alpha_j}^{0}\left(
| \lambda |r_{j}^{2}/2\right)
e^{-\frac{1}{4}| \lambda | r_{j}^{2}}e^{i\lambda a_j
r_{j}^{2}}r_j^{1+(z-1)\frac{\gamma}{n}} dr_j. \label{diag2}
\end{equation}
Thus, it is enough to study the integral
\[
\int\limits_{0}^{\infty }\eta_{1}(r^{2})L_{\alpha_1}^{0}\left(|\lambda|r^{2}/2\right) e^{-\frac{1}{4} | \lambda|r^{2}}
e^{i\lambda a_1 r^{2}} r^{1+(z-1)\frac{\gamma}{n}} dr,
\]
where $a_1 \in \mathbb{R}$ and $\eta_1 \in C_{c}^{\infty}(\mathbb{R})$. We make the change of variable $\sigma = |\lambda|r^{2}/2$ in such an integral to obtain
\[
\int\limits_{0}^{\infty }\eta_{1}(r^{2})L_{\alpha_1}^{0}\left(|\lambda|r^{2}/2\right) e^{-\frac{1}{4} | \lambda|r^{2}}
e^{i\lambda a_1 r^{2}} r^{1+(z-1)\frac{\gamma}{n}} dr
\]
\[
= 2^{-\frac{(n+1) \gamma}{(2+\gamma)n}}| \lambda |^{-\left( 1+\frac{(z-1)\gamma}{2n}\right)} \int\limits_{0}^{\infty }\eta_{1}\left( \frac{2\sigma }{
| \lambda |}\right) L_{\alpha_1}^{0}\left( \sigma \right) e^{-
\frac{\sigma }{2}}e^{i2sgn(\lambda ) a_1 \sigma }\sigma
^{\frac{(z-1)\gamma}{2n}}d\sigma
\]
\[
= 2^{-\frac{(n+1) \gamma}{(2+\gamma)n}} | \lambda |^{-\left( 1+\frac{(z-1)\gamma}{2n} \right) }\left( F_{\alpha_1 ,\beta}G_{\lambda }\right) \widehat{\left. {}\right. }
(-2sgn(\lambda ) a_1)
\]
\[
= 2^{-\frac{(n+1) \gamma}{(2+\gamma)n}} \left\vert \lambda \right\vert ^{-\left( 1+\frac{(z-1)\gamma}{2n}\right) }(\widehat{F_{\alpha_1 ,\beta}}\ast \widehat{G_{\lambda }}
)(-2sgn(\lambda ) a_1)
\]
where $$F_{\alpha_1, \beta}(\sigma ):=\chi _{(0,\infty )}(\sigma
)L_{\alpha_1}^{0}\left( \sigma \right) e^{-\frac{\sigma }{2}}\sigma
^{\beta },$$ with $\beta =\frac{ (z-1) \gamma }{2n}$, and
$$G_{\lambda }(\sigma ):=\eta _{1}\left( \frac{
2\sigma }{\left\vert \lambda \right\vert }\right).$$
Now
\begin{equation}
\left\vert ( \widehat{F_{\alpha_1, \beta}}\ast \widehat{G_{\lambda
}})(-2sgn(\lambda ) a_1)\right\vert \leq \left\Vert
\widehat{F_{\alpha_1 ,\beta}}\ast \widehat{G_{\lambda }}\right\Vert
_{\infty }\leq \left\Vert \widehat{F_{\alpha_1 ,\beta}}\right\Vert
_{\infty }\left\Vert \widehat{ G_{\lambda }}\right\Vert
_{1}=\left\Vert \widehat{F_{\alpha_1 ,\beta}}\right\Vert _{\infty
}\left\Vert \widehat{\eta _{1}}\right\Vert _{1}. \label{fb}
\end{equation}
So it is enough to estimate $\left\Vert
\widehat{F_{\alpha_1 ,\beta}}\right\Vert _{\infty }$. Since
$$
\widehat{F_{\alpha_1 ,\beta}}(\xi )=\int\limits_{0}^{\infty }\sigma
^{\beta }L_{\alpha_1}^{0}\left( \sigma \right) e^{-\sigma \left(
\frac{1}{2}+i\xi \right) }d\sigma,
$$
from Lemma 6, with $n=1$, $k= \alpha_1$ and $\beta =\frac{ (z-1) \gamma }{2n}$ we obtain
$$\widehat{F_{\alpha_1 ,\beta}}(\xi ) = \frac{\Gamma(\beta+1)}{\left(\frac{1}{2}+i\xi \right)^{\beta+1}}
\sum_{j+l= \alpha_1} \frac{\Gamma(-\beta+j)\Gamma(\beta+1+l)}{\Gamma(-\beta)\Gamma(\beta+1)} \frac{w^{l}}{j!l!},$$
to take modulo in this expression and since $\left\vert w \right\vert = 1$ it follows that
$$\left\vert \widehat{F_{\alpha_1 ,\beta}}(\xi ) \right\vert \leq
\frac{\Gamma(-Re(\beta)) \Gamma(Re(\beta)+1)}{\left\vert
\left(\frac{1}{2}+i\xi \right)^{\beta+1}\right\vert \left\vert
\Gamma(-\beta) \right\vert} \sum_{j+l=\alpha_1}
\frac{\Gamma(-Re(\beta)+j))
\Gamma(Re(\beta)+1+l)}{\Gamma(-Re(\beta))\Gamma(Re(\beta)+1)}
\frac{1}{j!l!}.$$
From Lemma 7, with $n=1$ and $k = \alpha_1$ we have
$$\sum_{j+l=\alpha_1}
\frac{\Gamma(-Re(\beta)+j))
\Gamma(Re(\beta)+1+l)}{\Gamma(-Re(\beta))\Gamma(Re(\beta)+1)}
\frac{1}{j!l!}=1.$$
So
$$\left\vert \widehat{F_{\alpha_1 ,\beta}}(\xi ) \right\vert \leq \frac{\Gamma(-Re(\beta))
\Gamma(Re(\beta)+1)}{\left\vert \left(\frac{1}{2}+i\xi
\right)^{\beta+1}\right\vert \left\vert \Gamma(-\beta)
\right\vert} \leq \frac{\Gamma
\left(\frac{(n+1) \gamma}{(2+\gamma)n} \right) \Gamma
\left(\frac{2n-\gamma}{(2+\gamma)n}\right)}{(1/2)^{\frac{2n-\gamma}{(2+ \gamma)n}} e^{-\frac{Im(z) \pi \gamma}{4n}} \left\vert
\Gamma\left(\frac{(1-z)\gamma}{2n} \right) \right\vert}.$$
Finally, for $Re(z)=-\frac{2n-\gamma}{2+\gamma}$, we obtain
$$| \nu_{z,N}(k,\lambda) |\leq c_{n, \gamma} |I_{1-z}(-\lambda) \phi_{N}(\lambda)| | \lambda |^{-\left( 1+\frac{(z-1)\gamma}{2n}\right)n} \prod_{j=1}^{n} \| \widehat{F_{\alpha_j, \beta}} \|_{\infty} \| \widehat{\eta_{j}} \|_{1}$$
$$
\leq \frac{c_{n, \gamma} \,\, e^{\frac {Im(z) \pi
\gamma}{4}} \left\Vert H \right\Vert_{\infty} \left[ \Gamma
\left(\frac{(n+1) \gamma}{(2+\gamma)n} \right) \right]^{n} \left[\Gamma
\left(\frac{2n-\gamma}{(2+\gamma)n}\right) \right]^{n} \prod_{j=1}^{n} \left\Vert
\widehat{\eta_{j}} \right\Vert_{1}} {\left\vert \Gamma
\left(\frac{1-z}{2}\right) \right\vert \left\vert
\Gamma\left(\frac{(1-z)\gamma}{2n} \right) \right\vert^{n}}.$$
By (\ref{L22}) it follows, for $Re(z)=-\frac{2n-\gamma}{2+\gamma}$, that
\[
\left\Vert T_{N,z}f\right\Vert _{L^{2}(\mathbb{H}^{n})}\leq c_{n, \gamma} \,\, \frac{e^{\frac {Im(z) \pi
\gamma}{4}}}{\left\vert \Gamma \left( \frac{1-z}{2}\right) \right\vert \left\vert
\Gamma\left(\frac{(1-z)\gamma}{2n} \right) \right\vert^{n}}
\left\Vert f\right\Vert _{L^{2}(\mathbb{H}^{n})}
\]%
It is easy to see, with the aid of the Stirling formula (see \cite{stein4},
p. 326), that the family $\left\{ T_{N,z}\right\} $ satisfies, on the strip $%
-\frac{2n-\gamma}{2+\gamma} \leq Re(z) \leq 1$, the hypothesis of the complex interpolation theorem
(see \cite{stein2} p. 205) and so $T_{N,0}$ is bounded from $L^{p_{\gamma}}(\mathbb{H}^{n})$ into $L^{p_{\gamma}'}(\mathbb{H}^{n})$ uniformly on $N$, where
$\left( \frac{1}{p_{\gamma}}, \frac{1}{p_{\gamma}'} \right) = C_{\gamma}$,
then letting $N$ tend to infinity, we obtain that the operator $T_{\mu_{\gamma}}$ is
bounded from $L^{p_{\gamma}}(\mathbb{H}^{n})$ into $L^{p_{\gamma}'}(\mathbb{H}^{n})$.
$\square$

\qquad

We re-establish Theorem 1 and Theorem 2 in the following

\begin{theorem} Let $\mu_{\gamma}$ be a fractional Borel measure as in Theorem 1. Then the interior of $E_{\mu_{\gamma}}$ is the open trapezoidal region with vertices $(0,0)$, $(1,1)$, $D$ and $D'$. Moreover $C_{\gamma}$ and the closed  segments joining $D'$ with $(0,0)$ and $D$ with $(1,1)$ except maybe $D$ and $D'$ are contained in $E_{\mu_{\gamma}}$.
\end{theorem}

\qquad

\textit{Proof of Theorem 3.} We consider, for each $N \in \mathbb{N}$ fixed, the analytic family $\{ U_{N, z} \}$ of operators on the strip $-n \leq Re(z) \leq 1$, defined by $U_{N, z}f = f \ast \widetilde{\mu}_{(1-z)\gamma} \ast J_{N,z}$ where $\widetilde{\mu}_{(1-z)\gamma}$ i given by (\ref{mu3}), $J_{N, z}$ by (\ref{jz}) and $U_{N, 0}f(x,t) \rightarrow U_{\widetilde{\mu}_{\gamma}}f(x,t) := (f \ast \widetilde{\mu}_{\gamma})(x,t)$ a.e.$(x,t)$ as $N \rightarrow \infty$. Proceeding as in the proof of Theorem 2 it follows, for $Re(z) = 1$, that $\| U_{N, z} \|_{1, \infty} \leq c \left| \Gamma (z/2) \right|^{-1}$. Also it is clear that, for $Re(z) = -n$, the kernel
$\widetilde{\mu}_{(1-z)\gamma} \ast J_{N,z} \in L^{1}(\mathbb{H}^{n})$ and it is a radial function. Now, our operator $\widehat{(\widetilde{\mu}_{(1-z)\gamma} \ast J_{N,z})}(\lambda)$ is diagonal, with diagonal entries $\nu_{N,z}(k, \lambda)$ given by
\[
\nu_{z,N}(k,\lambda )=\frac{k!}{(k+n-1)!} \int\limits_{0}^{\infty} (\widetilde{\mu}_{(1-z)\gamma} \ast J_{N,z})(s, \widehat{-\lambda}) L_{k}^{n-1}(|\lambda|s^{2}/2) e^{-|\lambda|s^{2}/4} s^{2n-1} ds
\]
\[
=\frac{k!}{(k+n-1)!}I_{1-z}(-\lambda
)\phi _{N}(\lambda )\int\limits_{0}^{\infty }\eta(s^{2}) L_{k}^{n-1}(|\lambda|s^{2}/2) e^{-|\lambda|s^{2}/4}  e^{i\lambda
s^{2m}}s^{2n-1+(1-z)\gamma }ds.
\]
Now we study this integral, the change of variable $\sigma = |\lambda| s^{2}/2$ gives
\[
\int\limits_{0}^{\infty }\eta(s^{2}) L_{k}^{n-1}(|\lambda|s^{2}/2) e^{-|\lambda|s^{2}/4}  e^{i\lambda
s^{2m}}s^{2n-1+(1-z)\gamma }ds
\]
\[
=2^{n-1+ \frac{(1-z)\gamma}{2}} |\lambda|^{-\left(n + \frac{(1-z)\gamma}{2}\right)} \int\limits_{0}^{\infty }\eta\left( \frac{2\sigma}{ |\lambda|} \right) L^{n-1}_{k}(\sigma) e^{-\frac{\sigma}{2}} e^{i (2\sigma)^{m} |\lambda|^{1-m} sgn(\lambda)}
\sigma^{n-1+ \frac{(1-z)\gamma}{2}} d\sigma.
\]
\[
=2^{n-1+ \frac{(1-z)\gamma}{2}} |\lambda|^{-\left(n + \frac{(1-z)\gamma}{2}\right)}\left(
\widehat{F_{k,\beta}}\ast \left( \widehat{G_{\lambda }}\ast
\widehat{R_{\lambda }}\right) \right) (0),
\]
where
$$F_{k,\beta}(\sigma ):=\chi _{(0,\infty )}(\sigma
)L_{k}^{n-1}\left( \sigma \right) e^{-\frac{\sigma }{2}}\sigma
^{\beta },$$ with $\beta =n-1+\frac{ (1-z)\gamma}{2}$,
$$G_{\lambda }(\sigma ):=\eta\left( 2\sigma / |\lambda|\right)$$ and
$$R_{\lambda }(\sigma )=\chi
_{(0,\left\vert \lambda \right\vert )}(\sigma
)e^{i2^{m}sgn(\lambda )\left\vert \lambda \right\vert ^{1-m}\sigma
^{m}}.$$
Now
\begin{equation}
\left\Vert \widehat{F_{k,\beta}}\ast \left( \widehat{G_{\lambda }}\ast \widehat{
R_{\lambda }}\right) \right\Vert_{\infty } \leq \left\Vert \widehat{F_{k,\beta}}
\right\Vert _{1}\left\Vert \widehat{G_{\lambda }}\right\Vert
_{1}\left\Vert \widehat{R_{\lambda }}\right\Vert _{\infty
}\label{fgr}
\end{equation}
So it is enough to estimate the right side of this inequality. From Lemma 6 and Lemma 7 we obtain
$$\left\vert \widehat{F_{k,\beta}}(\xi ) \right\vert  \leq \frac{\Gamma(n-1-Re(\beta))
\Gamma(Re(\beta)+1)}{\left\vert \left(\frac{1}{2}+i\xi
\right)^{\beta+1}\right\vert \left\vert \Gamma(n-1-\beta)
\right\vert} \frac{(n+k-1)!}{(n-1)!k!}.$$
Since $Re(z)=-n$,
$\gamma = \frac{2(m-1)}{(n+1)m}$ and $\beta =n-1+\frac{ (1-z)\gamma}{2}$ we have $Re(\beta) = n- \frac{1}{m},$ thus it follows that
\begin{equation}
\left\Vert \widehat{F_{k,\beta_{z}}} \right\Vert_{1} \leq \frac{c \,
e^{\frac{\vert Im(z) \vert \gamma \pi}{4}}} {\left\vert \Gamma
\left(\frac{z-1}{2}\gamma \right) \right\vert} \,
\frac{(n+k-1)!}{(n-1)!k!} \, \int\limits_{0}^{\infty}
\frac{1}{ \left(\frac{1}{4}+\xi^{2} \right)^{\frac{1}{2}
\left(n+1-\frac{1}{m} \right)}} d\xi, \label{estif}
\end{equation}
the last integral is finite for all $n \geq 1$ and all $m \geq 2$.
It is clear that $\left\Vert \widehat{G_{\lambda }}\right\Vert_{1}=\left\Vert
\widehat{\eta} \right\Vert_{1}$. Now, we estimate $\left\Vert \widehat{R_{\lambda }}\right\Vert _{\infty
}$. Taking account of Proposition 2, p. 332, in \cite{stein3} we note that
\begin{equation}
\left\vert \widehat{R_{\lambda }}(\xi )\right\vert =\left\vert
\int\limits_{0}^{\left\vert \lambda \right\vert }e^{i(2^{m}sgn(\lambda
)\left\vert \lambda \right\vert ^{1-m}\sigma ^{m}-\xi \sigma )}d\sigma
\right\vert \leq
\frac{C_{m}}{\left\vert \lambda \right\vert ^{\frac{1-m}{m}}}
\label{estir}
\end{equation}
where the constant $C_m$ does not depend on $\lambda$. Then, for $Re(z)=-n$,
from (\ref{fgr}), (\ref{estif}) and (\ref{estir}) we obtain
$$|\nu_{N,z}(k, \lambda)| \leq c \, \frac{k!}{(k+n-1)!} \, |I_{1-z}(-\lambda) \phi_{N}(\lambda)| |\lambda|^{-\left(n + \frac{(1+n)\gamma}{2} \right)}\left\Vert \widehat{F_{k,\beta}}\ast \left( \widehat{G_{\lambda }}\ast \widehat{
R_{\lambda }}\right) \right\Vert_{\infty }$$
$$\leq c_{n,m} \Vert H
\Vert_{\infty} \left\Vert \widehat{\eta} \right\Vert_{1}
\frac{ e^{\frac{\vert Im(z) \vert \gamma \pi}{4}}} {\left\vert \Gamma
\left(\frac{1-z}{2} \right) \right\vert \left\vert \Gamma
\left(\frac{z-1}{2}\gamma \right) \right\vert} \int\limits_{0}^{\infty}
\frac{1}{ \left(\frac{1}{4}+\xi^{2} \right)^{\frac{1}{2}
\left(n+1-\frac{1}{m} \right)}} d\xi.
$$
By (\ref{L22}) it follows, for $Re(z)=-n$, that
\[
\| U_{N,z} f \|_{L^{2}(\mathbb{H})^{n}} \leq c_{n,m} \, \frac{ e^{\frac{\vert Im(z) \vert \gamma \pi}{4}}} {\left\vert \Gamma
\left(\frac{1-z}{2} \right) \right\vert \left\vert \Gamma
\left(\frac{z-1}{2}\gamma \right) \right\vert} \|f\|_{L^{2}(\mathbb{H})^{n}}.
\]
It is clear that the family $\{ U_{N,z} \}$ satisfies, on the strip $-n \leq Re(z) \leq 1$, the hypothesis of the complex interpolation theorem. Thus $U_{N,0}$ is bounded from $L^{\frac{2n+2}{2n+1}}(\mathbb{H}^{n})$ into $L^{2n+2}(\mathbb{H}^{n})$ uniformly in $N$, and letting $N$ tend to infinity we conclude that the operator $U_{\widetilde{\mu}_{\gamma}}$ is bounded from $L^{\frac{2n+2}{2n+1}}(\mathbb{H}^{n})$ into $L^{2n+2}(\mathbb{H}^{n})$ for $n\in \mathbb{N}$. Finally, the theorem follows from the restrictions that appear in (\ref{restricciones}).
$\square$

\end{document}